\documentclass[12pt]{article}

\usepackage{amsmath}
\usepackage{amsthm}
\usepackage{amsfonts}
\usepackage{amssymb}
\usepackage[all]{xy}
\usepackage{pb-diagram}
\usepackage{comment}
\usepackage[pdftex,breaklinks,colorlinks,filecolor=blue,linkcolor=blue,citecolor=blue,urlcolor=blue,hypertexnames=true,plainpages=false]{hyperref}

\usepackage{color}

\newcommand\R{\mathbb{R}}
\newcommand\reals{\mathbb{R}}

\newcommand\Q{\mathbb{Q}}
\newcommand\Z{\mathbb{Z}}

\newcommand\GRS{\mathfrak{\Gamma}}

\newcommand\tensor{\otimes}
\newcommand\Diff{\textup{Diff}}

\newcommand\boundary{\partial}
\newcommand\del{\partial}
\newcommand\im{\mathrm{im}}
\newcommand\id{\mathrm{Id}}
\newcommand\wt[1]{\widetilde{#1}}
\newcommand\ol{\overline}

\newcommand\coker{\textup{coker}}
\newcommand\xra{\xrightarrow}
\newcommand{\cyl}{\mathrm{cyl}}
\newcommand{\fr}{\mathrm{fr}}

\theoremstyle{plain}
\newtheorem{theorem}{Theorem}[section]
\newtheorem{lemma}[theorem]{Lemma}
\newtheorem{corollary}[theorem]{Corollary}

\newtheorem{conjecture}[theorem]{Conjecture}

\theoremstyle{definition}

\theoremstyle{remark}
\newtheorem{remark}[theorem]{Remark}

\advance \topmargin by -\headheight
\advance \topmargin by -\headsep
     
\evensidemargin \oddsidemargin
\title{The derivative map for diffeomorphism of disks:\\ An example}
\author{Diarmuid Crowley, Thomas Schick and Wolfgang Steimle}
\date{\today}

\begin{document}
\maketitle

%
%


\begin{abstract}
We prove that the derivative map $d \colon \Diff_\del(D^k) \to \Omega^kSO_k$, defined
by taking the derivative of a diffeomorphism, can induce a nontrivial map on homotopy groups.
Specifically, for $k = 11$ we prove that the following homomorphism is non-zero:
$$ d_* \colon \pi_5\Diff_\del(D^{11}) \to \pi_{5}\Omega^{11}SO_{11} \cong \pi_{16}SO_{11} $$
As a consequence we give a counter-example to a 
conjecture of Burghelea and Lashof and so give an example of a 
non-trivial vector bundle $E$ over a sphere
which is trivial as a topological $\R^k$-bundle
(the rank of $E$ is $k=11$ and the base sphere is $S^{17}$.)

The proof relies on 
a recent result of Burklund and Senger which determines the homotopy 17-spheres
bounding $8$-connected manifolds,
the plumbing approach to the Gromoll filtration due to Antonelli, Burghelea and Kahn,
and an explicit construction of low-codimension embeddings of certain homotopy spheres.
\end{abstract}

\section{Introduction}
The derivative map 
\[d \colon \Diff_\del(D^k) \to \mathrm{Map}((D^k, \del D^k), (SO_k, \id)) \simeq \Omega^k SO_k, \quad f \mapsto (x\mapsto D_x f)\]
is a basic invariant of the diffeomorphism group of the $k$-disk; in fact the first order approximation in the embedding calculus approach to the diffeomorphism group. 
While the rationalisation of $d$,
$d_\Q \colon \Diff_\del(D^k)_\Q \to (\Omega^kSO_k)_\Q$,
is null-homotopic,
as we explain in Section \ref{s:cr},
much less is known about the derivative map $d$ integrally.  For example, to the best of our knowledge,
it was not yet known whether the map induced by $d$ on homotopy groups,
\[d_* \colon \pi_i\Diff_\del(D^k) \to \pi_{i+k} SO_k,\]
was ever non-trivial.
Burghelea and Lashof  showed that 
$d_*$ vanishes 
for $i = 0, 1$.  At odd primes $p$, they also showed that $d_* = 0$ provided $i<k{-}3$
and they made a conjecture equivalent to the claim 
that this holds for $p = 2$ as well \cite[Conjecture p.\,\,40]{B-L}.
Burghelea and Lashof also report A'Campo informing them about a proof that $d_* = 0$ for $i = 2$ 
(however a written proof has not appeared). 

Using smoothing theory, or an explicit geometric construction we introduce here, the map $d_*$ admits an interpretation as describing the normal bundle of certain homotopy spheres embedded in euclidean space. Combining this interpretation with recent results of Burklund--Senger and the refined plumbing construction of Antonelli--Burghelea--Kahn, we obtain a counterexample to the
conjecture of Burghelea and Lashof.

In more detail, in \cite{A-B-K} Antonelli, Burghelea and Kahn constructed families of diffeomorphisms of the disk using a pairing
\[\sigma \colon \pi_p SO_{q-a} \otimes \pi_q SO_{p-b} \to \pi_{a+b+1}\Diff_\del(D^{p+q-a-b-1}) \]
($0\leq a\leq q, 0\leq b\leq p$), refining Milnor's plumbing pairing (see below).  Then we have:

\begin{theorem} \label{thm:main}
Let $\xi\in\pi_8SO_6 \cong \Z/ 24$ be a generator.
 The image of $\sigma(\xi,\xi)$ under the derivative map
\[d_* \colon \pi_5\Diff_\del(D^{11}) \to  \pi_{16}SO_{11}\]
is non-zero.
\end{theorem}

Using Morlet's smoothing theory isomorphism, the derivative map $d_*$ on $\pi_k$ is identified with the boundary map 
\[ \del \colon \pi_{n+k+1}PL_k/O_k \to \pi_{n+k} SO_k\]
of the 
fibration sequence $SO_k \to SPL_k \to PL_k/O_k$
(and this allows for our interpretation of 
\cite[Conjecture p.\,\,40]{B-L} in terms of the derivative map). We conclude that the map $SO_{11}\to SPL_{11}$ is not injective on $\pi_{16}$.  More specifically, if $\tau_{11}\colon S^{11}\to BSO_{11}$ represents the tangent bundle of the $11$-sphere
and $f\colon S^{17}\to S^{11}$ represents the unique non-trivial homotopy class, we have:

\begin{corollary} \label{cor:main}
The pull-back $f^*\tau_{11}$ is a non-trivial vector bundle 
which becomes trivial as an $\R^{11}$-bundle (even when considered as a bundle with structure group $SPL_{11}$). 
\end{corollary}

To the best of our knowledge, this is the first example of a non-trivial
vector bundle over a sphere
which is known to be trivial as a topological $\R^n$-bundle.  
Milnor famously gave examples
of non-trivial vector bundles over Moore spaces, 
for example the Moore space $M(\Z/7, 7) = S^7 \cup_7 D^8$, which are trivial as 
$\R^n$-bundles \cite[Lemma 9.1]{M}.  
These examples are stable bundles over $4$-connected
spaces and so the vector bundles are trivial as piecewise linear bundles too.

\subsection*{Acknowledgements}
We would like to thank Sand Kupers for helpful comments on an earlier draft.

\section{Proofs} \label{s:proofs}
In this section we give the proofs of Theorem \ref{thm:main} and Corollary \ref{cor:main}.
We first recall Gromoll's map $A = C \circ \lambda$ \cite{G},
\[A\colon \pi_{n-k}\Diff_\del(D^k) \xrightarrow{~\lambda~} \pi_0\Diff_\del(D^n) 
\xrightarrow{~C~} \Theta_{n+1},\]
where $\Theta_{n+1}$ is the group of homotopy $(n{+}1)$-spheres. The first map $\lambda$ includes fiberwise diffeomorphisms of $D^{n-k}\times D^k$ into all diffeomorphisms
and the second $C$ uses a diffeomorphism of $D^n \subset S^n$ as a datum to clutch 
two $(n{+}1)$-disks to a homotopy sphere.%
\footnote{The map $C$ is denoted $\Sigma$ in \cite{C-S}.}

\begin{lemma}\label{lem:embedding_of_homotopy_sphere}
For any $x\in \pi_{n-k}\Diff_\del(D^k)$, the homotopy sphere $A(x)\in \Theta_{n+1}$ admits an embedding into $\mathbb R^{n+k+1}$ whose normal bundle is classified (up to possible sign) by $d_*(x)\in \pi_n SO_k \cong \pi_{n+1} BSO_k$. 
\end{lemma}

We will offer
two proofs of this result; one by an explicit geometric construction and a more abstract one by the classification of smoothings through Rourke-Sanderson's theory of block bundles.

Next we recall that Milnor constructed exotic spheres by plumbing linear disk bundles and taking the boundary sphere; this construction gives rise to a pairing
\[\sigma_M\colon \pi_p SO_q \otimes \pi_q SO_p \to \Theta_{p+q+1}.\]
By \cite[Proposition 3.1]{A-B-K}, the pairing of Antonelli, Burghelea and Kahn  refines this pairing in the sense that the following diagram is commutative
\begin{equation}\label{eq:commsquare}\xymatrix{
 \pi_p SO_{q-a} \otimes \pi_q SO_{p-b} \ar[rr]^(0.475){\sigma} \ar[d]
   && \pi_{a+b+1}\Diff_\del(D^{p+q-a-b-1})\ar[d]^A
 \\
 \pi_p SO_q \otimes \pi_q SO_p \ar[rr]^(0.55){\sigma_M} 
   && \Theta_{p+q+1},
}\end{equation}
where the map on the left is the tensor product of the canonical stabilisations.
We now consider the homotopy $17$-sphere 
$$\Sigma_{\xi, \xi} := A(\sigma(\xi,\xi)),$$
recalling that $\xi \in \pi_8SO_6 \cong \Z/24$ denotes a generator.
By the commutativity of \eqref{eq:commsquare} and the definition of $\sigma_M$, $\Sigma_{\xi, \xi}$ is the boundary of an $8$-connected compact $18$-manifold and so
by a recent result of Burklund and Senger \cite[Theorem 1.4]{B-S} its image under the 
normal invariant map $\Theta_{17}\to \coker(J_{17})$ must be either $0$ or $[\eta\eta_4]$. 
We will show:

\begin{lemma}\label{lem:Sigma}
The homotopy sphere $\Sigma_{\xi, \xi}$ represents $[\eta\eta_4]\in \coker(J_{17})$.
\end{lemma}

We deduce from this that every embedding $\Sigma_{\xi, \xi} \hookrightarrow S^{28}$
has a non-trivial normal bundle. Indeed, recall that the map $\Theta_{17}\to \coker(J_{17})$ is obtained by embedding a
homotopy $17$-sphere into some euclidean space with trivial normal bundle, and
performing the Pontryagin-Thom collapse as to obtain an element in
$\pi_{17}^s$ which is well-defined modulo the image of $J$. Now, assuming by
contradiction that $\Sigma_{\xi, \xi}$ embeds into $S^{28}$
with trivial normal bundle, then $[\eta\eta_4]$ would have a
representative in $\pi_{28}S^{11}$. However, this contradicts the computations of Toda
\cite[Theorem 2.17 \& Proposition 2.20]{T} on the stabilisation map
$\pi_{28}S^{11} \to \pi_{17}^s$, which we display below:
%
%
\[ 
\begin{array}{ccccccccccc}
\pi_{28}S^{11} & = &
\Z/2((\eta^2 \rho)_{11}) & \oplus & \Z/2((\mu_{17})_{11})
& \oplus & \Z/2((\nu \kappa)_{11}) \\
\downarrow & & \downarrow  & & \downarrow  & & \downarrow  \\
\pi_{17}^s & = & \Z/2(\eta^2 \rho) & \oplus & \Z/2(\mu_{17})
& \oplus & \Z/2(\nu \kappa) & \oplus & \Z/2(\eta \eta_4)
\end{array}
\]
Here the notation is such that an element $(\delta)_{11}$ stabilises to $\delta$ and
the stable class $\eta^2 \rho$ generates $\im(J_{17} \colon \pi_{17}(SO) \to \pi_{17}^s)$.
Lemma \ref{lem:embedding_of_homotopy_sphere} then implies that $d_*(\sigma(\xi, \xi)) \neq 0$, which concludes the proof of Theorem \ref{thm:main}, modulo the Lemmas \ref{lem:embedding_of_homotopy_sphere} and \ref{lem:Sigma}. 

To see Corollary \ref{cor:main}, we recall that the normal bundle $\nu(\Sigma_{\xi, \xi}\subset S^{28})=d_*(\sigma(\xi, \xi))\in \pi_{16}SO_{11}$ is non-zero but maps to $0\in \pi_{16}SPL_{11}$. On the other hand, by Antonelli \cite{A}, the normal bundle of every homotopy $17$-sphere embedded in euclidean space in codimension $12$ is zero.  Hence
\[\nu(\Sigma_{\xi, \xi} \subset S^{28})\in \ker(\pi_{17}BSO_{11}\to \pi_{17}BSO_{12}) = 
\im(\pi_{17}S^{11} \to \pi_{17}BSO_{11}),\]
where the last map is induced by the classifying map of the tangent bundle of the $11$-sphere.

It remains to prove the Lemmas \ref{lem:embedding_of_homotopy_sphere} and \ref{lem:Sigma}.

\begin{proof}[Proof of Lemma \ref{lem:embedding_of_homotopy_sphere}]
Choose a smooth map $\psi\colon D^{n-k}\times D^k\to D^k$
  representing the class $[\psi]\in \pi_{n-k}(\Diff_\partial(D^k))$  (i.e.~for $x\in
  D^{n-k}$ we have that $\psi_x:=\psi(x,-)\in \Diff_\partial(D^k)$, and $\psi_x=\id_{D^k}$ for $x\in\partial
  D^{n-k}$). Then $\lambda([\psi])$ is represented by 
  \[\Psi\colon D^k\times D^{n-k}\to D^k \times D^{n-k}, \quad (x,y)\mapsto (x,\psi(x,y))\]
  and $A([\psi])$ is represented by the homotopy sphere $\Sigma_\Psi^{n+1}$ obtained by glueing two copies of $D^{n+1}$ along the boundary using
  the diffeomorphism $\Psi$. Note also that the image of $[\psi]$ under the derivative map is represented by $d\psi\colon D^{n-k}\times D^k\to Gl_k(\reals)$ with $d\psi(x,y) =
  D_y\psi_x$.

  For technical reasons, we actually assume without loss of generality that
  the maps are the identity maps in a neighborhood of the boundaries.

  We construct an explicit embedding $\iota_\psi\colon \Sigma^{n+1}_\Psi\hookrightarrow S^{n+k+1}$ of $\Sigma^{n+1}_\Psi$, compute the normal bundle of this embedding,
  and show explicitly that it is obtained by clutching with $d\psi$.

As might be expected, given that all our data is on disks (and trivial near
the boundary of the disks), we actually produce an interesting embedding
$\iota_\psi$ of
$D^{n+1}=D^{n-k}\times D^k\times [0,1]$ into $D^{n+k+1}=D^{n-k}\times
D^k\times D^k\times [0,1]$ which has standard form
near the boundary, and then obtain an embedding of $\Sigma^{n+1}_\Psi$
by glueing with a standard embedding of $D^{n+1}$ into $D^{n+k+1}$ in the
appropriate way.

The desired embedding $\iota_\psi$ is explicitly given by
\begin{equation*}
  \iota_\psi\colon D^{n-k}\times D^k\times [0,1] \to D^{n-k}\times D^k\times
  D^k\times [0,1], \quad (x,y,t)\mapsto (x, \alpha(t) y, \beta(t) \psi_x(y), t);
\end{equation*}
here, $\alpha,\beta\colon [0,1]\to [0,1]$ are smooth maps such that
$\alpha(t)=1$ for $t<0.6$ and $\alpha(t)=0$ for $t>0.9$, and
$\beta(t)=\alpha(1-t)$.

This is evidently a smooth embedding whose image we denote by $S_\psi$, and we let $\del S_\psi=\iota_\psi(\boundary D^{n+1})$. Then $\del S_\psi\subset \boundary D^{n+k+1}$: To see this, observe that if either the
$x$ or the $t$-coordinate is in the boundary, then the first or forth
coordinate of the image point is so, too. For each $t\in [0,1]$, then
$\alpha(t)=1$ or $\beta(t)=1$. If $y\in\boundary D^k$, then therefore
either the second or the third component of the image point is in the boundary
(or both). As $\boundary (D^{n-k}\times D^k\times [0,1])$ is the union of those
points with at least one component in the boundary, this proves the claim.

We also note that the subset $\boundary S_\psi\subset \boundary D^{n+1}$ is in fact independent of $\psi$ (as $\psi$ is fixed to be the identity
map near the boundary).
Let us identify this image set $\boundary S_\psi$
with $S^{n}=\boundary(D^{n-k}\times D^k\times [0,1])$ via the restriction of
$\iota_{\id}$ to $\boundary(D^{n-k}\times D^k\times [0,1])$.

Then $\iota_\psi|\colon \boundary(D^{n-k}\times D^k\times [0,1])\to
\boundary(D^{n-k}\times D^k\times [0,1])$ is supported on the disk
$D^{n-k}\times D^k\times \{1\}$, where it is given by $\Psi$. Therefore, we can glue two copies of $D^{n-k}\times D^k\times D^k\times
[0,1]$ along the boundary by the identity map to obtain $S^{n+k+1}$, and the
embeddings $\iota_\psi$ in one copy and $\iota_\id$ in the other glue together
to form the desired embedding of $\Sigma^{n+1}_\Psi$ into $S^{n+k+1}$.

Strictly speaking, one has to round the corners off to get an actual smooth
embedding. This can easily be achieved, as $\psi$ is the identity in a
neighborhood of the boundaries. We omit spelling out the somewhat cumbersome
details.

It remains to compute the normal bundle of the embedding. To do this, we
first compute the 
differential 
\[D\iota_\psi\colon (D^{n-k}\times D^k\times [0,1])\times
(\reals^{n-k}\oplus \reals^k\oplus \reals) \to S_\psi\times \reals^{n-k}\oplus
\reals^k\oplus\reals^k\oplus\reals\]
to be given in each fiber by 
\begin{equation*}
  D_{(x,y,t)}\iota_\psi =
  \begin{pmatrix}
    1 &  0 & 0\\
    0 & \alpha(t) & \alpha'(t) y\\
    \beta(t)\partial_x \psi & \beta(t) d\psi &\beta'(t)\psi_x(y)\\
    0 &0 & 1
  \end{pmatrix}.
\end{equation*}
We obtain an explicit trivialization of the normal bundle of this embedding
via the fiberwise linear map
\begin{equation*}
  \nu \colon D^{n-k}\times D^k\times [0,1] \times \reals^k \to S_\psi\times
 ( \reals^{n-k}\oplus \reals^k\oplus \reals^k\oplus \reals)
\end{equation*}
covering $\iota_\psi$ and described by the matrix
\[\nu_{(x,y,t)} = 
  \begin{pmatrix}
   0 \\ -\beta(t) (d\psi)^{-1}  \\ \alpha(t) \\ 0
 \end{pmatrix}.
\]
To observe that this really describes the normal bundle, for dimension reasons
we just have to check
that the image of $\nu$ intersects the tangent bundle of $S$, i.e.~the image
of $D\psi$, trivially. It is clear that $\nu(v)$ can only be equal to a
tangent vector of the form $(0,\alpha(t)w,\beta(t)d\psi(w),0)$ for
$w\in\reals^k$. This implies $\alpha(t)v=\beta(t)d\psi(w)$ and
$-\beta(t)v=\alpha(t)d\psi(w)$; two equations imply $\alpha(t)^2v=-\beta(t)^2v$ and
finally (as $\alpha(t)^2+\beta(t)^2>0$) then $v=0$ and then also
$w=0$.  It follows that the image 
of $\nu$ represents the normal bundle of $S_\psi$ in $D^{n+k+1}$.

For the other half disk which produces the embedding of $\Sigma_\Psi^{n+1}$
into $S^{n+k+1}$, we obtain a trivialization of the normal bundle by the
same recipe, replacing $\psi$ by $\id$. We observe then that we obtain the
global normal bundle by glueing these two explicitly chosen normal subbundles
of $TD^{n+k+1}$ along the boundary, where they coincide. The trivializations differ precisely on the half disk $\iota(D^{n-k}\times D^k\times
\{1\})$, and there they differ by the derivative map $d\psi$. On the other
half disk, the two trivializations coincide.

Consequently, the normal bundle of the embedding $\iota_\psi$  is obtained by
clutching with $d\psi$, precisely as claimed, and the lemma is proved.  
\end{proof}

\begin{remark}
  It is tempting to hope that the explicit geometric construction of $d_*$ as the
  normal bundle of the embedding $\iota_\psi$ can be used to get some new information about
  $d_*$. On the other hand, the information obtained by the formulas in the
  proof given above seems rather limited.  
  At least in the case where $\psi$ lies in the image of $\sigma$, we present Conjecture \ref{conj:1}
  below on $d_* \circ \sigma$.
\end{remark}

\begin{proof}[Proof of Lemma \ref{lem:Sigma}]
Let $A_{18}^8$ denote the group of rel.\ boundary bordism classes of $8$-connected
$18$-manifolds with boundary a homotopy sphere. 
According to \cite[\S17]{W2} and \cite[Theorem 2 (5)]{W1}, we have an isomorphism 
\begin{equation}\label{eq:computationWall}
 A_{18}^8 \to \Z/2 \oplus \Z/2,
\quad [W] \mapsto (\Phi(\varphi_W), \varphi_W(\chi_W)). 
\end{equation}
Here $\varphi_W\colon H_9(W)\to \Z/2$ is a quadratic refinement of the intersection form, obtained by representing a homology class by an embedded sphere and taking its normal bundle:
This gives rise to a quadratic map 
$$\alpha_W\colon H_9(W)\to \pi_8SO_9,$$ 
where the stabilisation map $S \colon \pi_8SO_9 \to \pi_8SO = \Z/2$ is split
with kernel $\Z/2$;
from $\alpha_W$ we obtain a quadratic map $\varphi_W$ 
with values in  $\Z/2=\ker(S)$ through fixing a splitting of $\pi_8SO_9$.
The first component of \eqref{eq:computationWall} is then the Arf invariant of 
$\varphi_W\colon H_9(W)\to \Z/2$, and the second component is the value 
of $\varphi_W$ on any integral lift of the element 
\[\chi_W \in H_9(W; \pi_8SO)\cong H_9(W; \Z/2)\]
which is the Poincar\'e dual of the element $S\alpha_W \in \mathrm{Hom}(H_9(W), \Z/2) = H^9(W; \Z/2)$.

Let $S^3(\xi) \in \pi_8SO_9$ be the image of $\xi \in \pi_8SO_6$ under the inclusion $SO_6 \to SO_9$.
By the commutativity of \eqref{eq:commsquare}, $\Sigma_{\xi,\xi}$ is the boundary of the Milnor plumbing $W$ of $S^3(\xi) \in \pi_8SO_9$ with itself, and we compute 
$\varphi_W(\chi_W)$ as follows:
With $H_9(W)=\Z(x)\oplus \Z(y)$, the normal bundles obtained from representing $x$ and $y$ by embeddings are both given by $S^3(\xi)$, so we conclude 
$\varphi_W(x)=\varphi_W(y)$.
Furthermore, $S^3(\xi)$ stabilises to a generator of $\pi_8SO$, by Lemma \ref{lem:xi} below.
Thus, $S \alpha_W$ maps both $x$ and $y$ to a generator and so $\chi_W = x+y$.
We now compute that 
$$ \varphi_W(\chi_W) = \varphi_W(x + y) = \varphi_W(x) + \varphi_W(y) + \rho_2(\lambda(x, y)) = 1,$$
where $\rho_2(\lambda)$ denotes the mod 2 reduction of the intersection form. 

Now, taking the homotopy sphere
on the boundary defines a homomorphism
\begin{equation}\label{eq:del} \del \colon A_{18}^8 \to \Theta_{17}.\end{equation}
From the short exact sequence
\[0 \to bP_{18}(=\Z/2) \to \Theta_{17}\to \coker(J_{17})\to 0\]
and \cite[Theorem 1.4]{B-S} we see that the image of the map $\del$ from \eqref{eq:del} consists of precisely 4 different elements, so the map $\del$ is injective. Each of the $bP$-spheres is the boundary of a manifold $P$ which satisfies $S\alpha_P=0$ and therefore $\varphi_P(\chi_P)=0$, and it follows that the element $W$ from above must map, under $\del$, to  a non-$bP$-sphere, which then represents $[\eta\eta_4]$ in view of \cite[Theorem 1.4]{B-S}.
\end{proof}

\begin{lemma}\label{lem:xi}
The map $\Z/24\cong \pi_8SO_6\to \pi_8SO\cong \Z/2$ is surjective.
\end{lemma}

\begin{proof}
By \cite[Theorem 1.4]{L}, $(\Z/2)^3 \cong \pi_8SO_8\to \pi_8SO\cong \Z/2$ is
onto and therefore has a kernel of 4 elements. (We refer to \cite{Ke} for the
computation of the relevant homotopy groups.) On the other hand,
$(\Z/2)^2\cong \pi_8SO_7\to \pi_8SO_8$ is injective (its cokernel injects into
$\pi_8(S^7)\cong\Z/2$) and so has an image of 4 elements. These two subgroups
do not coincide: Since the maximal number of pointwise
  linearly independent vector fields on $S^9$ is $1$ \cite[Theorem 1.1]{Adams}, the tangent bundle of $S^9$ defines an element in $\pi_8SO_8$ that is not in the image of $\pi_8SO_7$ but maps to $0\in \pi_8SO$. 

Therefore, $(\Z/2)^2\cong \pi_8SO_7\to \pi_8SO$ is surjective and has a kernel
of precisely two elements; similarly the image of $\Z/24\cong\pi_8SO_6\to
\pi_8SO_7\cong (\Z/2)^2$ consists of precisely two elements (its cokernel
injects into $\pi_8S^6 \cong \Z/2$), and we are left to show that these two subgroups do not agree. To see this, we consider the element  $a:=(2\gamma)_7\eta_7$ where $(2\gamma)_7$ is a generator of $\Z\cong \pi_7SO_7$ and $\eta_7\colon S^8\to S^7$ is the nontrivial class: By \cite[Theorem 1.4]{L}, $(2\gamma)_7$ stabilizes to an element divisible by 2 and so $a$ is in the kernel of the stabilization; and it does not lift to $\pi_8SO_6$ by the commutativity of the following diagram with  exact rows:
\[\xymatrix{
 & \pi_7SO_7 \ar[r] \ar[d]^{\eta_7} & \pi_7S^6 \ar[r] \ar[d]^{\eta_7}_\cong & \pi_6SO_6\;(=0)
 \\
 \pi_8SO_6 \ar[r] &\pi_8SO_7 \ar[r] & \pi_8S^6
}\]
\end{proof}

We conclude this section by giving the promised second proof of Lemma \ref{lem:embedding_of_homotopy_sphere}. To this end we recall from \cite[\S 6]{R-S1} that a smoothing of $S^{n+1}$ in $S^{n+k+1}$ consists of a smooth manifold $W$ and a PL homeomorphism $H\colon W\to S^{n+k+1}$, such that $\Sigma:=H^{-1}(S^{n+1})\subset W$ is a smooth submanifold, and such that $H$ is concordant to the identity smoothing of $S^{n+k+1}$; and recall the group $\GRS^k_{n+1}$ of concordance classes of such smoothings.%
\footnote{The group $\GRS^k_{n+1}$ is denoted $\Gamma^{k}_{n+1}$ in \cite{R-S1}.  We have
used different notation, to avoid confusion with the notation
$\Gamma^{n+1}_k$ for the subgroups of the Gromoll filtration. }
We note that up to diffeomorphism, $W$ is a standard sphere mapping to $S^{n+k+1}$ by a PL homeomorphism concordant to the identity, so that elements of $\GRS^k_{n+1}$ are represented by PL homeomorphisms $H\colon S^{n+k+1}\to S^{n+k+1}$ which are concordant to the identity (i.e., orientation preserving). Note also that $\Sigma$ is a homotopy $(n{+}1)$-sphere, oriented through the PL homeomorphism $h:=H\vert_{\Sigma}$, which is smoothly embedded into $S^{n+k+1}$. 

There are two obvious homomorphisms out of $\GRS^k_{n+1}$, 
\[ \Theta_{n+1} \xleftarrow{~F~} \GRS^k_{n+1} \xrightarrow{~\nu~} \pi_{n+1}BSO_k,\]
the left one mapping the class of $H$ to the diffeomorphism class of $\Sigma$ and the right one to the classifying map of the normal bundle of $\Sigma\subset S^{n+k+1}$
(where, as usual, we identify a homotopy sphere up to homotopy equivalence with a standard sphere using the given orientation).
Then, Lemma \ref{lem:embedding_of_homotopy_sphere} is clearly implied by the following result:

\begin{lemma}\label{lem:d_and_nu}
There exists a group homomorphism $B\colon \pi_{n-k}\Diff_\del(D^k)\to \GRS^k_{n+1}$ such that the following diagram commutes up to possible signs:
\[\xymatrix{
 &\pi_{n-k}\Diff_\del(D^k) \ar[r]^(0.6){d_*} \ar[d]^{B} \ar[ld]_A & \pi_n SO_k \ar[d]^\cong\\
 \Theta_{n+1} &\GRS^k_{n+1} \ar[l]_(0.45)F \ar[r]^(0.425)\nu & \pi_{n+1} BSO_k
}\]
\end{lemma}

\begin{proof}[Proof of Lemma \ref{lem:d_and_nu}]
We consider the following diagram:
\begin{equation}\label{eq:big_diagram}\xymatrix{
\pi_0\Diff_\del(D^n) \ar[dd]^{M_{n*}} \ar@/_12ex/[ddd]_C^\cong  & 
\pi_{n-k}\Diff_\del(D^k) \ar[l]_\lambda \ar[rr]^d \ar[d]^{M_{k*}} && 
\pi_n SO_k \ar[d]^=\\
& \pi_{n+1}PL_k/O_k \ar[d]^{\wt i} \ar[dl]_S \ar[rr]^{\del} && 
\pi_nSO_k  \ar[d]^=\\
\pi_{n+1}PL/O &
\pi_{n+1}\wt{PL}_k/O_k  \ar[l]_{\wt S}^{} \ar[rr]^{\del} &&
\pi_nSO_k \\
\Theta_{n+1} \ar[u]_{\Psi}^\cong  & 
\GRS^k_{n+1} \ar[u]_\cong \ar[l]_{F} \ar[rr]^{\nu}  &&
\pi_{n+1} BSO_k \ar[u]_\cong
}\end{equation}
Here $\Psi$ is the map which sends a homotopy sphere $\Sigma$ to the element
represented by the tangent PL microbundle of the mapping cylinder
$\cyl(h\colon \Sigma \to S^{n+1})$ of an orientation preserving PL
homeomorphism $h$, along with its linear structure induced by the smooth
structure of $\Sigma$ on the $\Sigma$ end of the cylinder and its canonical trivialization at $S^{n+1}$ end. The map $\Psi$ is
an isomorphism by smoothing theory \cite[II Theorem 4.2]{H-M}. 
The map $\GRS^k_{n+1} \to \pi_{n+1}\widetilde{PL}_k/O_k$ is an isomorphism by
\cite[Corollary 6.7]{R-S1}: 
It is defined by sending the class of $(H,h)\colon (S^{n+k+1}, \Sigma^{n+1})\to (S^{n+k+1}, S^{n+1})$ to 
the normal block bundle $\nu_\cyl$ of $\cyl(h)$ inside $\cyl(H)$ along with its linear reduction at the $\Sigma^{n+1}$ end of the cylinder and its canonical trivialization at the other end. 
Finally, the map $\tilde S$ is obtained from the fact that the inclusion $PL\to \widetilde{PL}$ is an equivalence 
\cite[Corollary 5.5 (ii)]{R-S3};
 that is, there is no essential difference between stable PL (micro-)bundles and stable block bundles. 

We claim that the left lower square of \eqref{eq:big_diagram} is commutative up to sign. To see this, 
we may assume, increasing $k$ if necessary, that the normal block bundle $\nu_\cyl$ is given by a PL microbundle. Then, the sum of the two composites, applied to $[(H,h)]$, is represented by the direct sum microbundle $T\cyl(h)\oplus \nu_\cyl$ over $\cyl(h)$ along with its linear reduction at the front end and its canonical trivialization at the other end. 
But now, we have an isomorphism $T\cyl(h)\oplus\nu_\cyl\cong T\cyl(H)\vert_{\cyl(h)}$ of microbundles which extends isomorphisms $T\Sigma \oplus \nu_{\Sigma\subset S^{n+k+1}}\cong TS^{n+k+1}\vert_{\Sigma}$ and $TS^{n+1}\oplus \nu_{S^{n+1}\subset S^{n+k+1}}\cong TS^{n+k+1}\vert_{S^{n+1}}$ of vector bundles.

Since $H$ is PL isotopic to the identity (being an orientation preserving PL homeomorphism of the sphere), we conclude that the sum of the two composite maps, applied to $[(H,h)]$, represents the zero element.  

All other parts of this diagram commute up to possible signs: The commutativity of the squares on the right and of the triangle in the middle follows from the definitions.
That $M_{n*} \circ \lambda = S \circ M_{k*}$ follows from \cite[Lemma 2.5]{C-S}
and that $M_{n*} = \Psi^\tau \circ C$ is proven in \cite[Lemma 2.7]{C-S}. 
The lemma now follows by a diagram chase.
\end{proof}

\section{Concluding remarks} \label{s:cr}

In this section we discuss some of background to our results and state a conjecture
about the map $d_* \circ \sigma$.

\medskip\noindent \textbf{1.} 
The homotopy fibre of $d \colon \Diff_\del(D^k) \to \Omega^kSO_k$
is the $H$-space $\Diff^{\fr}_\del(D^k)$ of framing-preserving diffeomorphisms.  It is
the loop space of the classifying space $B\Diff^{\fr}_\del(D^k)$, which features in the recent work of 
Kupers and Randal-Williams \cite{K-RW}
on the rational homotopy groups of $\Diff_\del(D^k)$.  We see that $d$ is rationally trivial because
the Alexander trick implies that $d$ becomes null-homotopic under
the natural map $\Omega^k SO_k \to \Omega^k SPL_k$.
It is well known that $(SO_k)_\Q$ is Eilenberg-MacLane, detected by the suspensions
of the rational Pontrjagin classes and rational Euler class.  Since these classes are defined on $(SPL_k)_\Q$, it follows that $(SO_k)_\Q$ is a homotopy retract of $(SPL_k)_\Q$.
If $\Omega^k_0X$ denotes the connected component of the constant map, then it follows
that $( \Omega^k_0 SO_k)_\Q \simeq \Omega^k_0 (SO_k)_\Q$ is a homotopy 
retract of $(\Omega^k_0 SPL_k)_\Q$, showing that the map 
$d\colon \Diff_\del(D^k) \to  \Omega^k_0 SO_k$ is rationally nullhomotopic.

\medskip\noindent \textbf{2.} The proof of Theorem \ref{thm:main} relies on the fact that the normal bundle of any embedding
$\Sigma_{\xi, \xi} \hookrightarrow S^{28}$ is non-trivial.  
Despite the elementary argument we give for this in Section \ref{s:proofs}, computing the normal bundle of an embedding of a homotopy sphere $g \colon \Sigma^{n+1} \hookrightarrow S^{n+k+1}$ is a subtle problem.  Provided one is in the meta-stable range $n < 2k{-}4$, Haefliger \cite{H} proved that the isotopy class of $g$ depends only on the diffeomorphism type of $\Sigma$, so that, in particular, the normal bundle is independent of the choice of embedding. Hsiang, Levine and Sczarba \cite{H-L-S} proved that the latter statement holds even for $n< 2k{-}2$, defined the homomorphism
%
%
$$ \phi^k_{n+1} \colon \Theta_{n+1} \to \pi_{n+1}BSO_k, 
\quad \Sigma \mapsto \nu(\Sigma \subset S^{n+k+1}) \quad (n<2k{-}2)$$
%
and proved that $\phi^{13}_{16} \neq 0$; i.e.\ the exotic $16$-sphere embeds into $S^{29}$ with non-trivial normal bundle.  
Then Antonelli \cite{A} made a systematic study of normal bundles of homotopy spheres in the meta-stable range, which includes the statement that $\phi^{11}_{17} \neq 0$.

\medskip\noindent \textbf{3.} Concerning A'Campo's claim that $d_*$ vanishes for $i = 2$, we note that, since
$\phi^{13}_{16} \neq 0$, Lemma \ref{lem:embedding_of_homotopy_sphere} entails that if A'Campo's claim
holds, then the exotic $16$-sphere does not lie in the image of the map
$A \colon \pi_2\Diff_\del(D^{13}) \to \Theta_{16} = \Z/2$.  
This is consistent with computations we have made for the refined plumbing pairing
$$ \sigma \colon \pi_8SO_6 \otimes \pi_7SO_8 \to \pi_2\Diff_\del(D^{13}),$$
which show that $A \circ \sigma = 0$, even though 
$\sigma_M \colon \pi_8SO_7 \otimes \pi_7SO_8 \to \Theta_{16}$ is non-trivial,
a statement which can be deduced from \cite[Satz 12.1]{S}.



\medskip\noindent \textbf{4.} 
Finally, we present a conjectural description of the homomorphism
$$ d_* \circ \sigma \colon 
 \pi_p SO_{q-a} \otimes \pi_q SO_{p-b} \to \pi_{p+q}SO_{p+q-a-b-1}$$
in purely homotopy-theoretic terms. 

Let $h \colon \pi_iSO_j \to \pi_i(S^{j-1})$ be the map induced by the canonical projection 
$SO_j \to S^{j-1}$.  
For maps $f \colon W \to X$ and $f \colon Y \to Z$ let $f \ast g \colon W \ast Y \to X \ast Z$ be their {\em join}.  
Let $\del \colon \pi_{m+1}(S^k) \to \pi_mSO_k$ denote the boundary map in the homotopy
long exact sequence of the fibration $SO_k \to SO_{k+1} \to S^k$.
For compactness, we use the notation 
$$p' := p-b \quad \text{and} \quad q' := q-a$$
and let $\xi_1 \in \pi_pSO_{q'}$ and $\xi_2 \in \pi_qSO_{p'}$.
Then we have $h(\xi_1) \in \pi_p(S^{q'-1})$, $h(\xi_2) \in \pi_qS^{p'-1}$ and
$h(\xi_1) \ast h(\xi_2) \in \pi_{p+q+1}(S^{p'+q'-1})$, so that
$\del \bigl( h(\xi_1) \ast h(\xi_2) \bigr) \in \pi_{p+q}SO_{p'+q'-1}$.

In addition, we have the $J$-homomorphisms
$$ J_{p, q'} \colon \pi_pSO_{q'} \to \pi_{p+q'}S^{q'}
\quad \text{and} \quad
J_{q, p'} \colon \pi_qSO_{p'} \to \pi_{q+p'}S^{p'}$$
and we can suspend in the target of each of these to get the homomorphisms
$$ \Sigma^a \circ J_{p, q'} \colon \pi_pSO_{q'} \to \pi_{p+q}S^{q}
\quad \text{and} \quad
\Sigma^b \circ J_{q, p'} \colon \pi_qSO_{p'} \to \pi_{p+q}S^{p}.$$
We then take compositions with the maps induced by $\xi_i$, $i = 1, 2$ and the inclusions
$i_{p'} \colon SO_{p'} \to SO_{p'+q'-1}$ and 
$i_{q'} \colon SO_{q'} \to SO_{p'+q'-1}.$
Hence we have the homomorphisms
%
$$ 
\ol \xi_{2*} \colon
\pi_pSO_{q'} \xra{\Sigma^a \circ J_{p, q'}} \pi_{p+q}S^q \xra{\xi_{2*}} \pi_{p+q}SO_{p'}
\xra{i_{p'*}} \pi_{p+q}SO_{p'+q'-1}$$
and 
$$  
\ol \xi_{1*} \colon
\pi_qSO_{p'} \xra{\Sigma^b \circ J_{q, p'}} \pi_{p+q}S^p \xra{\xi_{1*}} \pi_{p+q}SO_{q'}
\xra{i_{q'*}} \pi_{p+q}SO_{p'+q'-1}.$$

\begin{conjecture} \label{conj:1}
Up to sign, the homomorphism
\[ d_* \circ \sigma \colon 
 \pi_p SO_{q'} \otimes \pi_q SO_{p'} \to \pi_{p+q}SO_{p'+q'-1} \]
is given by 
%
\[  d_*( \sigma(\xi_1, \xi_2)) = \del \bigl( h(\xi_1) \ast h(\xi_2) \bigr) + 
\ol \xi_{1*}(\xi_2) + \ol \xi_{2*}(\xi_1). \]
\end{conjecture}

We briefly discus Conjecture \ref{conj:1} in the light ot Theorem \ref{thm:main} and Corollary \ref{cor:main}.
For $\xi \in \pi_8SO_6$ a generator,
$h(\xi) \in \pi_8S^5 \cong \pi^s_3$ is again a generator and we choose $\xi$
so that $h(\xi) = \nu_5$.
Hence Conjecture \ref{conj:1} gives 
$d_*(\sigma(\xi, \xi)) = \del(\nu_5 \ast \nu_5) + 2\ol \xi_*(\xi)$.
Now $\pi_{16}S^8 \cong (\Z/2)^4$, which entails that $2 \ol \xi_*(\xi) = 0$ 
and the proof of Corollary \ref{cor:main} shows that $d_*(\sigma(\xi, \xi)) = \del(\nu^2_{11})$.
Since $\nu_5 \ast \nu_5 = \nu^2_{11}$, 
Conjecture \ref{conj:1} is consistent with Theorem \ref{thm:main} and 
Corollary \ref{cor:main},
with both giving the same non-zero expression for
$d_* \circ \sigma \colon \pi_8SO_6 \tensor \pi_8SO_6 \to \pi_{16}SO_{11}$.

\bigskip
\noindent
\emph{Diarmuid Crowley}\\
  \vskip -0.125in \noindent
  {\small
  \begin{tabular}{l}%
    School of Mathematics and Statistics\\
    University of Melbourne, Australia\\
    \textsf{email:~dcrowley@unimelb.edu.au} \\ 
    \textsf{web:~www.dcrowley.net}
  \end{tabular}}

\vskip 0.75cm
\noindent
\emph{Thomas Schick}\\  
\vskip -0.125in \noindent
{\small
  \begin{tabular}{l}%
Mathematisches Institut, \\Universit\"at G{\"o}ttingen, Germany \\
    \textsf{email:~thomas.schick@math.uni-goettingen.de}\\
    \textsf{web:~www.uni-math.gwdg.de/schick}
  \end{tabular}}

\vskip 0.75cm
\noindent
\emph{Wolfgang Steimle}\\  
\vskip -0.125in \noindent
{\small
  \begin{tabular}{l}%
Institut f\"ur Mathematik\\ Universit\"at Augsburg, Germany \\
    \textsf{email:~wolfgang.steimle@math.uni-augsburg.de}\\
    \textsf{web:~www.uni-augsburg.de/de/fakultaet/mntf/math/prof/diff/team/wolfgang-steimle}
  \end{tabular}}

\end{document}